\documentclass[times]{article}

\usepackage{moreverb}
\usepackage{cite}
\usepackage{longtable}
\usepackage{graphicx}
\usepackage{amsmath}
\usepackage{amscd}
\usepackage{amssymb}
\usepackage{wrapfig}
\usepackage{mathrsfs}
\usepackage{amsthm}
\usepackage[matrix,arrow,curve]{xy}

\newcommand\BibTeX{{\rmfamily B\kern-.05em \textsc{i\kern-.025em b}\kern-.08em
T\kern-.1667em\lower.7ex\hbox{E}\kern-.125emX}}

\newtheorem{definition}{Definition}[section]

\newtheorem{theorem}{Theorem}[section]
\newtheorem{lemma}{Lemma}[section]

\newtheorem{note}[theorem]{Remark}
\newtheorem{ex}[theorem]{Example}

\newcommand{\lmod}{\mathrm{-mod}}
\newcommand{\rmod}{\mathrm{mod-}}
\newcommand{\Rmod}{\mathrm{Mod-}}
\newcommand{\Lex}{\mathcal{L}ex}
\newcommand{\Rex}{\mathcal{R}ex}
\newcommand{\coker}{\mathrm{coker\ }}

\newcommand{\colim}{\mathrm{colim\ }}
\newcommand{\Hom}{\mathrm{Hom }}
\newcommand{\RHom}{\mathbf{R}\mathrm{Hom }}
\newcommand{\End}{\mathrm{End }}

\newcommand{\pd}{\mathbf{pd.}}

\newcommand{\id}{\mathbf{id}}
\newcommand{\Ext}[1]{\mathrm{Ext}^{#1}}

\date{}
\begin{document}
\title{On derived equivalences of exact structures on the category of representations $k[x]/(x^n)$}
\author{Nikita~Monchenko}
\maketitle
\begin{abstract}
From the perspective of homological algebra, exact categories are interesting because they possess derived categories. Typically, an additive category admits many exact structures. If the original category is sufficiently nice (for example, quasi-abelian), a classification of exact structures can be given in terms of certain Serre subcategories in the category of contravariant functors. It turns out that the derived categories of different exact structures are functorially related. This work is concerned with sufficient conditions for the derived equivalence of different exact structures on certain categories.
\end{abstract}

\section{Introduction}
Exact categories were first introduced by Quillen in 1972 for the definition of higher K-theory. The formalism of exact categories allows extending the apparatus of homological algebra from abelian categories to a wider class of additive categories. The question of classifying exact structures has been widely discussed in the literature. In \cite{Rump}, it was proven that on any additive category there exists a maximal exact structure containing all smaller ones. Later, an explicit description of the maximal exact structure was obtained for quasi-abelian categories, as well as for categories with split idempotents (see \cite{Cr} for details).

Intermediate exact structures, which are neither maximal nor minimal, are of interest. In \cite{En}, a complete classification of exact structures on a category with split idempotents was given in terms of Serre subcategories of the category of contravariant functors $\mathcal{F}unc(\mathcal{E}^{op}, \mathcal{A}b)$ (Theorem \ref{Classif}). However, from a practical point of view, this theorem turns out to be difficult to apply, since the problem of finding all Serre subcategories is extremely complex. Nevertheless, in some cases, for example, for Krull-Schmidt categories of finite type, it is possible to obtain an explicit description of all exact structures on the category in its internal terms, using Auslander theory and almost split sequences. Furthermore, the problem of classifying exact structures on a category is closely related to the study of properties of the right/left abelian envelopes of the category, introduced by A.I. Bondal and A. Bodzenta in \cite{BB}. It turns out that if an exact structure $\epsilon$ is contained in an exact structure $\epsilon '$, then the abelian envelope of the category with the exact structure $\epsilon '$ is a quotient category of the abelian envelope of this category with the exact structure $\epsilon$.

The main goal of the present work is to study the question of when two exact structures on a given additive category are derived equivalent. This question has hardly been discussed in the literature so far. Unfortunately, a criterion for the derived equivalence of exact structures has not yet been obtained. Nevertheless, the case of the category of representations $\rmod k[x]/(x^n)$ was studied, and a necessary and sufficient condition for derived equivalence was formulated for it (Theorem \ref{main}). These questions turned out to be closely related to the existence of spherical objects in the abelian envelope of the category with the minimal exact structure.

The study of abelian envelopes of exact categories is also interesting from the point of view that the abelian envelope of the maximal exact structure always has a global dimension not exceeding 2 (Theorem \ref{dim}), i.e., it is smooth. This allows considering it as a noncommutative resolution of the original category, possibly non-smooth (i.e., of infinite global dimension). Sometimes a smooth resolution can be given by the abelian envelope of a smaller exact structure.

\section{Exact categories}
Let $\mathcal{E}$ be an additive category (i.e., a category with a zero object and all finite biproducts. This is sufficient to define an abelian group structure on morphisms). We will call a \textit{conflation} or a kernel-cokernel pair a triple
$$X\xrightarrow{i} Y \xrightarrow{d} Z,$$
where $i = \ker d$, $d = \coker i$. The morphism $d$ will be called a deflation (or admissible epimorphism), and $i$ an inflation (or admissible monomorphism).
\begin{definition}
    An exact structure on $\mathcal{E}$ is a set $S$ of conflations satisfying three dual pairs of axioms:
    \begin{enumerate}
        \item For any $A\in \mathcal{E}$, $\id_A$ is an inflation, i.e., $S$ contains a conflation of the form 
        $$A\xrightarrow{\id} A \xrightarrow{} 0.$$
        \item For any $A\in \mathcal{E}$, $\id_A$ is a deflation, i.e., $S$ contains a conflation of the form 
        $$0\xrightarrow{} A \xrightarrow{\id} A.$$
        \item The composition of inflations is an inflation.
        \item The composition of deflations is a deflation.
        \item Let $i: \ X\rightarrow Y$ be an inflation, $f: \ X\rightarrow Z$ an arbitrary morphism. Then there exists a Cartesian square:
        $$
        \xymatrix{
        X \ar[r]^{i} \ar[d]_f& Y \ar[d]\\
        Z\ar[r]^{i'} & Y \coprod_X Z \\
        }
        $$
        and $i'$ is an inflation. 
        \item Let $d: \ X\rightarrow Y$ be a deflation, $f: \ Z\rightarrow Y$ an arbitrary morphism. Then there exists a co-Cartesian square:
        $$
        \xymatrix{
        Z\prod_Y X \ar[r]^{d'} \ar[d]_f& Z \ar[d]_{f}\\
        X\ar[r]^{d} & Y \\
        }
        $$
        and $d'$ is an inflation. 
    \end{enumerate}
\end{definition}

\begin{definition}
    A pair $(\mathcal{E}, S)$ of an additive category and an exact structure on it is called an \textit{exact category}.
\end{definition}
Henceforth, wherever this does not lead to ambiguity and the exact structure is clear from the context, we will call the category $\mathcal{E}$ exact.
\begin{ex}
    On any additive category, there exists a minimal exact structure, consisting of split conflations of the form
    $$X\rightarrow X\oplus Z \rightarrow Z,$$
    where the morphisms are the standard inclusion and projection into the direct sum. Indeed, for any two objects, we can write a lift along the zero morphism:
    $$\xymatrix{
    X \ar[r]\ar[d] & X\oplus Z \ar[d] \ar[r] & Z\ar[d]\\
    X \ar[r]^{\id} & X \ar[r] & 0\\
    }$$
\end{ex}

\begin{theorem}[\cite{Rump}]
    On any additive category, there exists a maximal exact structure by inclusion, containing all smaller ones.
\end{theorem}
\begin{ex}
    In an abelian category, the standard exact structure is maximal. Indeed, any monomorphism is a kernel and can be completed to a short exact sequence, just as any epimorphism is a cokernel and can also be completed to a short exact sequence. It is obvious that the axioms of an exact category are satisfied.
\end{ex}
\begin{definition}
    An additive category $\mathcal{A}$ is called \textit{quasi-abelian} if any morphism has a kernel and a cokernel, any lift of a cokernel along an arbitrary morphism exists and is a cokernel, and any base change of a kernel along an arbitrary morphism exists and is a kernel.
\end{definition}
\begin{definition}
    A morphism $f$ in a quasi-abelian category is called \textit{strict} if the natural morphism $\coker(\ker f)\rightarrow \ker(\coker f)$ is an isomorphism.
\end{definition}
\begin{ex}
    Kernel-cokernel pairs where both morphisms are strict form an exact structure on a quasi-abelian category. In \cite{Cr}, it is shown that this exact structure is maximal.
\end{ex}

Thus, on any additive category, there exist two possibly coinciding exact structures: the minimal (trivial or split) and the maximal. 
\begin{ex}
    The category of vector spaces over an arbitrary field admits only one exact structure. Indeed, it is abelian, and any short exact sequence of vector spaces is split.
\end{ex}

We will say that a morphism $f$ in an exact category is \textit{admissible} if it can be represented as a composition $f = i \circ d$, where $d$ is an admissible epimorphism and $i$ is an admissible monomorphism. 

Without delving into technical details, we will explain how the derived category of an exact category is defined, while also recalling standard definitions from homological algebra. A more detailed construction of the derived category of an exact category is described in \cite{Buh}.

For any additive category, the notion of a chain complex makes sense. A \textit{chain complex} over an additive category $\mathcal{E}$ is a collection of objects $\{K^i\}_{i =-\infty}^{\infty}$ and differentials $d^i :\  K^i\rightarrow K^{i+1}$ satisfying $d^{i+1}\circ d^i = 0$. Morphisms of complexes are defined naturally. It is obvious that if the original category $\mathcal{E}$ was exact/abelian, then the category of complexes $\mathcal{K}(\mathcal{E})$ will also be exact/abelian, respectively. Let $K^\bullet$ be a complex with admissible differentials. Then split each differential:
$$\xymatrix{
...\ar@{->>}[rd] \ar[rr]^{d^{i-2}}&&K^{i-1}\ar@{->>}[rd] \ar[rr]^{d^{i-1}}&& K^i\ar@{->>}[rd] \ar[rr]^{d^{i}}&& K^{i+1}\ar@{->>}[rd] \ar[rr]^{d^{i+1}}&& ...\\
 & Z^{i-1} \ar@{^{(}->}[ru] && Z^{i} \ar@{^{(}->}[ru] && Z^{i+1} \ar@{^{(}->}[ru] && Z^{i+2} \ar@{^{(}->}[ru]\\
}$$

A morphism of complexes $f^\bullet: \ K^\bullet \rightarrow L^\bullet$ is called \textit{null-homotopic} if there exist morphisms $h^i :\ K^i \rightarrow L^{i-1}$ such that $f^i = h^{i+1}\circ d_{K}^i + d_L^{i-1}\circ h^i$; two morphisms $f^\bullet, g^\bullet: \ K^\bullet \rightarrow L^\bullet$ are called \textit{homotopic} if $f^\bullet - g^\bullet \sim 0$. A complex is called contractible if the identity morphism on it is null-homotopic. Complexes $K^\bullet$ and $L^\bullet$ are called homotopy equivalent if there exist $f^\bullet:\ K^\bullet \rightarrow L^\bullet$ and $g^\bullet:\ L^\bullet \rightarrow K^\bullet$ such that $fg\sim \id_L$, $gf \sim \id_K$. The localization of the category of complexes $\mathcal{K}(\mathcal{E})$ is called the \textit{homotopy category} $\mathcal{H}(\mathcal{E})$.

We say that a complex $K^\bullet$ is \textit{exact} or \textit{acyclic} if each triple $Z^i\rightarrow K^i \rightarrow Z^{i+1}$ is a conflation. 

Let $f^\bullet: K^\bullet \rightarrow L^\bullet$ be a morphism of complexes. Its cone, as in standard homological algebra, is the complex $C(f)^\bullet$ with terms $C(f)^i = K^{i+1}\oplus L^{i}$ and differential
$$d_{C(f)}^i = \begin{pmatrix}
    d_K^{i+1} & 0\\
    f^i & d_L^{i}\\
\end{pmatrix}$$

Unfortunately, in an exact category, the notion of cohomology does not generally make sense. Nevertheless, the notion of a quasi-isomorphism can be defined. A morphism of complexes $f^\bullet: \ K^\bullet \rightarrow L^\bullet$ is called a \textit{quasi-isomorphism} if its cone $C(f)^\bullet$ is acyclic.

\begin{definition}
    The \textit{derived category} $\mathcal{D}(\mathcal{E})$ is the localization of the homotopy category $\mathcal{H}(\mathcal{E})$ with respect to quasi-isomorphisms. The bounded derived category $\mathcal{D}^*(\mathcal{E})$, $*\in \{+,-,\flat\}$ is the full subcategory of $\mathcal{D}(\mathcal{E})$ whose objects are quasi-isomorphic to bounded below/above/both-sided complexes, respectively.
\end{definition}

Just as in the abelian case, the categories $\mathcal{H}(\mathcal{E})$ and $\mathcal{D}(\mathcal{E})$ are triangulated. 

\begin{ex}
    Let $\mathcal{A}$ be an abelian category with enough projective objects. Consider the full subcategory $\mathcal{P}$ of projective objects of $\mathcal{A}$. It possesses a natural exact structure induced by the exact structure on $\mathcal{A}$: a triple $P_0 \rightarrow P_1 \rightarrow P_2$ is a conflation in $\mathcal{P}$ if and only if the corresponding sequence in $\mathcal{A}$ is a short exact sequence. A classical statement from homological algebra allows us to easily conclude that $\mathcal{D}^-(\mathcal{P})$ and $\mathcal{D}^-(\mathcal{A})$ are equivalent as derived categories of exact categories. Note that subcategories of projective objects of abelian categories were historically the first examples of exact categories; they were considered by Quillen when defining the Q-construction for algebraic K-theory. 
\end{ex}

To conclude our overview of the basic principles of the theory of exact categories, we give a formulation of a theorem that is an analogue of the famous Freyd-Mitchell theorem for abelian categories.

\begin{theorem}[Gabriel, Quillen]
    Let $\mathcal{E}$ be a skeletally small exact category. Then there exists an abelian category $\mathcal{A}$ and a faithful, full, and exact functor $i: \ \mathcal{E}\rightarrow \mathcal{A}$ (i.e., sending conflations to short exact sequences), reflecting exactness (i.e., the preimage of a short exact sequence is a conflation), and $i(\mathcal{E})$ is closed under extensions in $\mathcal{A}$. 
\end{theorem}

In other words, any exact category is a full exact subcategory of some abelian category. The idea of the proof lies in observing that the class of deflations of an exact structure defines a Grothendieck topology $\tau$ on the category $\mathcal{E}$, turning the category into a site $(\mathcal{E},\tau)$. It turns out that this topology is subcanonical, i.e., any representable functor $F(-) \cong \Hom(-,A)$ is a sheaf. Then the required abelian category is the category of sheaves $\mathcal{A} = Sh(\mathcal{E})$, and the functor $i$ is the usual Yoneda embedding: $A \mapsto \Hom(-,A)$.

\section{Abelian envelopes}
As already mentioned above, the Gabriel-Quillen theorem guarantees that any exact category $\mathcal{E}$ can be embedded into some abelian category. In some cases, the ambient abelian category can be chosen in a canonical way. Henceforth, except for some technical lemmas, the exposition follows the article \cite{BB}.

\begin{definition}
    A functor $F:\ \mathcal{E}\rightarrow\mathcal{H}$ between exact categories is called \textit{right exact} if for any conflation $X\xrightarrow{i} Y \xrightarrow{d} Z$ in $\mathcal{E}$, the morphism $F(d)$ is a deflation and has a kernel $i': X'\rightarrow F(Y)$ such that
    $$\xymatrix{
    F(X) \ar[r]^{F(i)} \ar @{->>}[d]_{d'}& F(Y)\ar@{->>} [r]^{F(d)} & F(Z)\\
    X'\ar@{^{(}->}[ru]^{i'}& & \\
    }$$
    for some deflation $d':\ F(X)\rightarrow X'$ (in particular, $F(i)$ is an admissible morphism).
\end{definition}

In particular, if the target category $\mathcal{H}$ is abelian, the definition means that for any conflation $X\xrightarrow{i} Y \xrightarrow{Z}$ in $\mathcal{E}$, the sequence $F(X)\xrightarrow{F(i)} F(Y) \xrightarrow{d} F(Z) \xrightarrow{} 0$ is exact in $\mathcal{H}$.

The category of right exact functors between two exact categories will be denoted by $\Rex(\mathcal{E},\mathcal{H})$.

Dually, left exact functors between exact categories can be defined. Their category will be denoted by $\Lex(\mathcal{E},\mathcal{H})$.

\begin{definition}
    A \textit{right abelian envelope} of an exact category $\mathcal{E}$ is a pair $(\mathcal{A}_r(\mathcal{E}),i_\mathcal{E})$, where $\mathcal{A}_r(\mathcal{E})$ is an abelian category, $i_\mathcal{E}:\ \mathcal{E}\rightarrow \mathcal{A}_r(\mathcal{E})$ is a right exact functor, such that for any abelian category $\mathcal{B}$, the functor $i_\mathcal{E}$ induces an equivalence of categories $\mathcal{R}ex(\mathcal{E},\mathcal{B})\cong \Rex(\mathcal{A}_r(\mathcal{E}),\mathcal{B})$. In other words, for any right exact functor $F:\ \mathcal{E}\rightarrow \mathcal{B}$, there exists a unique up to natural isomorphism functor $\Tilde{F}:\ \mathcal{A}_r(\mathcal{E})\rightarrow \mathcal{B}$ such that the following diagram commutes (up to a canonical natural isomorphism):
    $$\xymatrix{
    \mathcal{E} \ar[r]^{i_\mathcal{E}} \ar [d]_F& \mathcal{A}_r(\mathcal{E}) \ar[dl]^{\Tilde{F}}\\
    \mathcal{B}& \\
    }$$
\end{definition}
Dually, a left abelian envelope can be defined.

From the universal property, it is obvious that the right abelian envelope, if it exists, is uniquely defined up to canonical equivalence of categories.

Consider the following example, proposed by J.-P. Schneiders in \cite{Sch}.
\begin{ex}
    Let $\mathcal{E}$ be a quasi-abelian category with its natural exact structure. The category $\mathcal{LH(E)}$ is a localization of the subcategory $\mathcal{K(E)}$ consisting of complexes of the form
    $$0\longrightarrow E \xrightarrow{f} F \longrightarrow 0,$$
    where $F$ is in degree zero, $f$ is a monomorphism, with respect to morphisms of complexes such that the square
    $$\xymatrix{
    E_1 \ar[r] & F_1 \\
    E_0 \ar[u]^{\delta_E} \ar[r]& F_0 \ar[u]^{\delta_F}\\
    }$$
    is Cartesian and co-Cartesian. The category $\mathcal{LH(E)}$ is abelian (\cite{Sch}, Corollary 1.2.21). There exists a faithful, full, and exact functor $i_\mathcal{E}: \mathcal{E}\rightarrow \mathcal{LH(E)}$, sending an object $X\in\mathcal{E}$ to the complex $$0\longrightarrow X \longrightarrow 0,$$ where $X$ is in degree zero. Moreover, there exists a functor $\pi:\ \mathcal{LH(E)}\rightarrow \mathcal{E}$, sending a complex $(E\xrightarrow{f} F)$ to $\coker f$. It is easy to see that an object $E\xrightarrow{f} F$ lies in the image of the embedding $i_\mathcal{E}$ if and only if $f$ is a strict monomorphism. Indeed, this is equivalent to the square 
    $$\xymatrix{
    0 \ar[r]& \coker f \\
    E \ar[r]^{f} \ar[u]& F\ar[u]\\
    }$$
    being Cartesian and co-Cartesian. In particular, if $\mathcal{E}$ was abelian, then the category $\mathcal{LH(E)}$ is equivalent to it. The embedding $i_\mathcal{E}$ induces an equivalence of unbounded derived categories $\mathcal{D(E)}$ and $\mathcal{D(LH(E))}$. Moreover, the category $\mathcal{LH(E)}$ is the right abelian envelope of the category $\mathcal{E}$.
\end{ex}

If it is known that an abelian envelope exists, an explicit description can be given.

\begin{definition}
    An object $X$ in a category $\mathcal{C}$ is called compact if for any filtered colimit along a diagram $F:\ \mathcal{I}\rightarrow \mathcal{C}$, there is an isomorphism $$\colim \Hom(X, F(i))\xrightarrow{\cong}\Hom(X, \colim F(i)).$$
\end{definition}

We need the following well-known result:
\begin{lemma}
    Let $R$ be an arbitrary associative algebra with unity. Then the subcategory of compact objects in $\rmod R$ coincides with the subcategory of finitely presented modules $f.p.(R)$.
\end{lemma}

Note that 
$$\Lex(\mathcal{A}_r(\mathcal{E})^{op},\mathcal{A}b)\cong \Rex(\mathcal{A}_r(\mathcal{E}),\mathcal{A}b^{op})\cong \Rex(\mathcal{E},\mathcal{A}b^{op})\cong \Lex(\mathcal{E}^{op},\mathcal{A}b)$$
Thus, composing this equivalence with the Yoneda embedding, we obtain that $\mathcal{A}_r(\mathcal{E})$ is a full subcategory of $\Lex(\mathcal{E}^{op},\mathcal{A}b)$. Denote this embedding by $I$.

\begin{theorem}[\cite{BB}, Lemma 4.3.]
    Let $\mathcal{E}$ be an exact category that has a right abelian envelope $\mathcal{A}_r(\mathcal{E})$. Then the embedding $i_\mathcal{E}$ is a faithful functor, and the image of the embedding $I$ coincides with the class of compact objects in $\Lex(\mathcal{E}^{op},\mathcal{A}b)$.
\end{theorem}

Nevertheless, an abelian envelope satisfying the above universal property is not guaranteed to exist. 

 \begin{definition}
        A module $M$ over an algebra with unity $R$ is called \textit{coherent} if every finitely generated submodule of it is finitely presented. An algebra $R$ is called coherent if it is coherent as a (right/left) module over itself. 
    \end{definition}

It is well-known that algebra is coherent iff the category $f.p.(A)$ is abelian. In fact, a category of finitely presented projective modules over non-coherent ring never has neither left nor right abelian envelope. Consider the category $\mathcal{P}(A)$ of finitely presented projective modules over $A$ with the induced exact structure. Then their abelian envelope should consist of finitely presented modules over $A$, but, as claimed before, this category is not abelian.

The following theorem provides a convenient tool for proving that a given abelian category is the abelian envelope of its full and exact subcategory:

\begin{theorem}[\cite{BB}]
\label{abenv}
    Let $\mathcal{A}$ be an abelian category, $\mathcal{E}\subset \mathcal{A}$ its full exact subcategory closed under kernels of epimorphisms. Also let any object of $\mathcal{A}$ be a quotient of an object from $\mathcal{E}$. Then $\mathcal{A}$ coincides with the right abelian envelope of $\mathcal{E}$. 
\end{theorem}

The formalism of abelian envelopes will allow us to give a convenient description of exact structures on some additive categories.

\section{Classification of exact structures}
As already mentioned earlier, any additive category admits 2 (possibly coinciding) exact structures: maximal and minimal (split). A natural question arises: are there other exact structures? Before giving a complete answer to this question, let us consider a naive example where the study of this question turns out to be simple. Let us first introduce a useful definition.

\begin{definition}
    An additive category $\mathcal{E}$ is called a category \textit{with split idempotents} if any idempotent $\varphi :\ X\rightarrow X$, (i.e., $\varphi^2 = \varphi$) has a kernel which splits off as a direct summand in $X$, i.e., $X\cong \ker \varphi \oplus \textrm{Im }\varphi$.
\end{definition}


Now, following \cite{En}, we give a classification of exact structures for a broad class of additive categories. All categories are assumed to be skeletally small and locally small.

Let $\mathcal{E}$ be an additive $k$-linear category with split idempotents. Consider the category $\Rmod \mathcal{E} = \mathcal{F}unc^{add}(\mathcal{E}^{op},k-Vect)$ -- the category of contravariant functors into the category of vector spaces over $k$. The category $\mathcal{E}$ is naturally embedded into $\Rmod \mathcal{E}$ via the Yoneda functor: $h:\ X\mapsto \Hom(-,X)$. Note that this embedding does not send conflations to short exact sequences.

At first, we remind to well-known lemmas.
\begin{lemma}
    The category $\Rmod \mathcal{E}$ is abelian with enough projective objects. The projective modules in $\Rmod \mathcal{E}$ are exactly the direct summands of $\bigoplus_{i\in I} h_{X_i}$. 
\end{lemma}
\begin{lemma}
\label{YE}
    The Yoneda embedding is a continuous functor, i.e., it preserves limits: $\lim(h_{X_\iota} )\cong h_{\lim {X_\iota}}$ along an arbitrary diagram $\iota: \ \mathcal{C}\rightarrow \mathcal{E}$ for which the limit exists.
\end{lemma}

\begin{definition}
    We say that a module $M\in \Rmod \mathcal{E}$ is \textit{finitely generated} if there exists $X\in \mathcal{E}$ such that $h_X$ covers $M$. The category of finitely generated modules is denoted by $\rmod \mathcal{E}$.
\end{definition}
Note that, generally speaking, this category is not abelian. Nevertheless, it is closed under extensions, hence it is exact. Indeed, if $h_X$ covers $M$, and $h_Y$ covers $N$, then their direct sum $h_X \oplus h_Y$ covers any extension $K$:
$$\xymatrix{
0 \ar[r] & h_X \ar [r]  \ar@{->>}[d] & h_X\oplus h_Y \ar[r]\ar@{->>}[d]& h_Y\ar[r] \ar@{->>}[d]& 0 \\
0 \ar[r] & M \ar[r] & K \ar[r] & N\ar[r] & 0}$$

From lemma \ref{YE}, in particular, it follows that the Yoneda embedding preserves all kernels. Then the following statement holds:
\begin{lemma}
    The following statements are equivalent:
    \begin{itemize}
        \item The sequence $X\xrightarrow{f} Y\xrightarrow{g} Z$ is a kernel-cokernel sequence in $\mathcal{E}$.
    \end{itemize} The sequence \begin{equation}
    \label{def}
    0 \rightarrow h_X\xrightarrow{h_f} h_Y\xrightarrow{h_g} h_Z \rightarrow M \rightarrow 0,
    \end{equation}
    where $M = \coker h_g$, is exact, $\Ext{0,1}(M,\mathcal{E}) = 0$, and $\pd M = 2$.
\end{lemma}
\begin{definition}
    A module $M$ satisfying the equivalent conditions of the lemma will be called \textbf{2-regular} or a \textbf{(covariant) defect} of the sequence $X \rightarrow Y \rightarrow Z$.
\end{definition}
\begin{proof}
    ($\Rightarrow$) First, we show that the sequence 
    $$0 \rightarrow h_X\xrightarrow{h_f} h_Y\xrightarrow{h_g} h_Z$$
    is exact. Exactness at the first term is obvious from the continuity of the Yoneda embedding. Next, let $h\in \Hom(B,Y) = h_Y$ lie in the kernel of $h_g$. Then $g\circ h = 0$, hence, by the definition of a kernel, $h$ factors uniquely through $X$, so it lies in the image of $h_f$. The exactness at the second term follows from the fact that the composition $g\circ f$ is zero.

    Now let $M = \coker h_g$. Applying the functor $\Hom(-,\mathcal{E})$, which is the composition of the Yoneda embedding and the Hom functor in the category $\Rmod \mathcal{E}$, to sequence (\ref{def}), we obtain:
    $$0 \rightarrow \Hom(M,\mathcal{E})\rightarrow \Hom(h_Z,\mathcal{E}) \rightarrow \Hom(h_Y,\mathcal{E}) \rightarrow \Hom(h_X,\mathcal{E})$$
    Again, due to the continuity of the contravariant Yoneda embedding, the map $\Hom(h_Z,\mathcal{E}) \rightarrow \Hom(h_Y,\mathcal{E})$ is a monomorphism, hence $\Hom(M,\mathcal{E}) = 0$. Similarly, applying $\Ext{1}(-,\mathcal{E})$ we have the exact sequence $0 \rightarrow \Ext{1}(M,\mathcal{E})\rightarrow \Ext{1}(h_Z,\mathcal{E}) = 0$. From this, it is obvious that $\Ext{1}(M,\mathcal{E}) = 0$, as required.

    ($\Leftarrow$) Since the projective dimension of $M$ is 2, there exists a projective resolution 
    $$0 \rightarrow h_X\xrightarrow{h_f} h_Y\xrightarrow{h_g} h_Z \rightarrow M \rightarrow 0$$
    for some objects $X$, $Y$, $Z\in \mathcal{E}$. All arguments in the previous step relied on necessary and sufficient conditions, so, repeating them verbatim and using that the Yoneda embedding is, in fact, an embedding, we obtain that the pair $X\xrightarrow{f} Y \xrightarrow{g} Z$ is a kernel-cokernel pair.
\end{proof}

Note that a dual statement holds for the contravariant Yoneda embedding. Denote $h^X = \Hom(X,-)$. It is easy to see that applying the functor $\Hom(-,\mathcal{E})$ to sequence (\ref{def}) yields:
$$0 \rightarrow h^Z\xrightarrow{h^g} h^Y\xrightarrow{h^f} h^X \rightarrow \Ext{2}(M,\mathcal{E}) \rightarrow 0$$

Denote by $Def(\mathcal{E})\subset \rmod\mathcal{E}$ the full subcategory of all defects of the category $\mathcal{E}$. Obviously, the functor $\Ext{2}(-,A)$ gives a duality equivalence between $Def(\mathcal{E})$ and $Def(\mathcal{E}^{op})\subset \mathcal{E}\lmod$.

Let $\epsilon$ be an exact structure on $\mathcal{E}$. Denote by $Def_\epsilon(\mathcal{E})\subset Def(\mathcal{E})$ the full subcategory whose objects are defects of conflations in the exact category $(\mathcal{E},\epsilon)$. This category will be called the category of defects of the exact structure. 
\begin{theorem}[\cite{En}, Theorem 2.7]
\label{Classif}
    \begin{enumerate}
        \item The subcategory $Def_\epsilon \mathcal{E}$ is a Serre subcategory in $\rmod A$. 
        \item Exact structures on the category $\mathcal{E}$ are in one-to-one correspondence with Serre subcategories $\mathcal{S}$ of the category $\rmod\mathcal{E}$ lying in the category $Def(\mathcal{E})$ and such that $\Ext{2}(\mathcal{S},\mathcal{E})$ is a Serre subcategory in the category $\mathcal{E}\lmod$. The inverse map is given as follows. To a Serre subcategory $\mathcal{S}$, we associate the exact structure
$$\epsilon(\mathcal{S}) = \{X\xrightarrow{f} Y \xrightarrow{g} Z \mbox{ -- a kernel-cokernel sequence such that }\coker(h_g)\in\mathcal{S}\}.$$
    \end{enumerate}
\end{theorem}

In the case where the category $\mathcal{E}$ is quasi-abelian, the defect category $Def(\mathcal{E})$ itself is a Serre subcategory in $\rmod \mathcal{E}$. Moreover, the symmetric condition that $\Ext{2} (\mathcal{S},\mathcal{E})$ is a Serre subcategory in $\mathcal{E}\lmod$ is automatically satisfied. Thus, for quasi-abelian categories, it suffices to study only Serre subcategories in $Def(\mathcal{E})$.

\section{Categories of finite type}
In the previous chapter, a description of all exact structures on an additive category with split idempotents was given. Unfortunately, in practice, the application of this classification is significantly hindered. However, consider a special case where the description of exact structures can be given in concrete and internal terms of the category $\mathcal{E}$.
\begin{definition}
    An object $A\in \mathcal{E}$ of an additive category is called an \textit{additive generator} if the category of objects that are direct summands of finite direct sums of $A$ coincides with $\mathcal{E}$.
\end{definition}

\begin{definition}
    An additive category $\mathcal{E}$ is called a category of \textit{finite type} if it possesses an additive generator.
\end{definition}

\begin{definition}
    An additive category $\mathcal{E}$ is called \textit{Krull-Schmidt} if every object decomposes into a finite direct sum of indecomposable objects, and the endomorphism rings of indecomposable objects are local.
\end{definition}
Note that this decomposition is unique up to isomorphism.

It is easy to see that a Krull-Schmidt category is of finite type if and only if it has a finite number of indecomposable objects. The additive generator of this category is their direct sum.

Let $\mathcal{E}$ be a $\Hom$-finite Krull-Schmidt category with a finite number of indecomposable objects $\{A_i\}_{i=0}^N$. Denote $E = \End(\bigoplus_{i=0}^N A_i)$ -- the endomorphism algebra of the additive generator of this category, and $\mathcal{A} = \rmod E$ -- the category of right modules over this algebra. From the Krull-Schmidt condition on the original category, the equivalence $\mathcal{E}\cong \mathcal{P}_\mathcal{A}$ follows automatically. The equivalence is given by the Yoneda embedding $E\mapsto \Hom(-,E) = h_E$. 

\begin{note}
    Under the assumptions of the previous paragraph, if the category $\mathcal{E}$ admits kernels of epimorphisms, then by Theorem \ref{abenv}, the category $\rmod E$ coincides with the right abelian envelope of the category $\mathcal{E}$. In particular, $\mathcal{D}^{-}\mathcal{E}\cong \mathcal{D}^{-}(\rmod E)$.
\end{note}

In the case where the category is of finite length and has a finite number of simple objects, any Serre subcategory is uniquely determined by the simple objects contained in it and is their closure under extensions. A natural question arises: which kernel-cokernel sequences correspond to the simple modules lying in the category $Def(\mathcal{E})$? We preface the answer to this question with a definition:
\begin{definition}
     In an additive category, a kernel-cokernel pair $A \xrightarrow{f} B \xrightarrow{g} C$ is called an \textbf{almost split} or \textbf{Auslander-Reiten (AR) sequence} if it is not split, but any morphism $\varphi: \ X\rightarrow C$ that is not a split epimorphism (i.e., does not admit a section $s$ such that $\varphi\circ s = \id_C$) factors through the object $B$:
    $$\xymatrix{
    A \ar[r]^{f}& B \ar[r]^{g}& C \\
     & X \ar[ru]_{\varphi} \ar @{.>}[u]& \\
    }$$
    and, dually, any morphism $\psi : \ A\rightarrow Y$ that is not a split monomorphism extends to a morphism from $B$:
    $$\xymatrix{
    &Y  & & \\
    A \ar[ru]^{\psi} \ar[r]^{f}& B \ar @{.>}[u] \ar[r]^{g}& C \\
    }$$
\end{definition}

It is obvious that an almost split sequence cannot end in a projective object and cannot begin with an injective object. 

\begin{lemma}
    Let $A\xrightarrow{i}B\xrightarrow{\pi} C$ be an almost split sequence. Then $A$ and $C$ are indecomposable.
\end{lemma}
\begin{proof}
    Suppose $C = \bigoplus_{i=0}^n C_i$, with $C_i$ indecomposable. Then each inclusion of a direct summand must factor through $B$, hence $C$ splits off as a direct summand in $B$, contradicting the definition.
\end{proof}

\begin{lemma}[Auslander, 1978]
    \begin{enumerate}
    Let $\mathcal{E}$ be a Krull-Schmidt category.
        \item Let $A_1\longrightarrow B_1 \longrightarrow C$ and $A_2\longrightarrow B_2 \longrightarrow C$ be two Auslander-Reiten sequences ending in the same indecomposable non-projective object $C$. Then they coincide, i.e., $A_1\cong A_2$ and $B_1\cong B_2$.
        \item Let $A\longrightarrow B_1 \longrightarrow C_1$ and $A\longrightarrow B_2 \longrightarrow C_2$ be two Auslander-Reiten sequences starting with the same indecomposable non-injective object $A$. Then they coincide, i.e., $A_1\cong A_2$ and $B_1\cong B_2$.
    \end{enumerate}
\end{lemma}

In what follows, we will be interested in the first historical case where Auslander-Reiten sequences exist. 

\begin{theorem}[Auslander, 1978]
\label{Seq}
    Let $A$ be an Artin algebra over a commutative Artinian ring $k$ (for example, a field), i.e., it is a finitely generated $k$-module. Then
    \begin{enumerate}
        \item For any indecomposable non-projective right $A$-module $M$, there exists an Auslander-Reiten sequence 
        $$0\longrightarrow A \longrightarrow B \longrightarrow M \longrightarrow 0$$
        \item For any indecomposable non-injective right $A$-module $M$, there exists an Auslander-Reiten sequence 
        $$0\longrightarrow M \longrightarrow B \longrightarrow C \longrightarrow 0$$
    \end{enumerate}
\end{theorem}
The conditions under which almost split sequences exist for any indecomposable non-projective or non-injective module can be weakened. The key factor in constructing almost split sequences is the locality of the endomorphism ring of the indecomposable object. For more details on this, see \cite{Aus}. We will say that a category $\mathcal{C}$ \textit{has Auslander-Reiten sequences} if Theorem \ref{Seq} holds in it.

Henceforth, for convenience and uniformity of notation for projective objects, we will denote the Yoneda embedding as $P_X = h_X$, emphasizing that the representable functor is a projective module.
\begin{theorem}
    A kernel-cokernel pair $X \xrightarrow{f} Y \xrightarrow{g} Z$ is an AR-sequence if and only if $M = \coker (P_g)$ and $N = \coker(P^f)\cong \Ext{2}(M,A)$ are simple modules.
\end{theorem}
\begin{proof}
    ($\Rightarrow$) Let $A \xrightarrow{f} B \xrightarrow{g} C$ be an AR-sequence. Suppose $M$ has a non-trivial submodule $M'$. Due to the Krull-Schmidt property, $M'$ has a projective resolution:
    \begin{equation*}
        \xymatrix{
        0 \ar[r] & P_X \ar[r] & P_Y \ar[r] & P_Z \ar[r] & M \ar[r] & 0 \\
         & P_{X'} \ar[u] \ar[r] & P_{Y'} \ar[u] \ar[r] & P_{Z'} \ar[u] \ar[r] & M' \ar[r] \ar @{^{(}->} [u] & 0 \\
        }
    \end{equation*}
    Since the Yoneda embedding is faithful and full, the morphism of projective resolutions induces a morphism $\varphi \ : \ Z'\rightarrow Z$. Due to the almost split nature of the sequence under consideration, it factors through $Y$ and induces a morphism $h_0 \ : \ P_{Z'} \rightarrow P_Y$. It is easy to see that this morphism extends to a homotopy of projective resolutions. We prove this by induction. Let $P_\bullet$ and $P'_\bullet$ be two projective resolutions with differentials $d_i\ : \ P_{i-1}\rightarrow P_i$ and $d'_i\ : \ P'_{i-1}\rightarrow P'_i$, and $f_i\ : \ P'_{i}\rightarrow P_i$ a morphism of projective resolutions. Suppose $n$ morphisms $h_i \ : \ P'_i \rightarrow P_{i-1}$ have been constructed such that $f_0 = d_0 \circ h_0$, and $f_i = d_i \circ h_i + h_{i-1}\circ d'_{i-1}$ for all $i<n$. Let $g_{n+1} = f_{n+1} - h_{n}\circ d'{n}$. Due to commutativity, it is obvious that $g_n$ is a morphism into $\ker d_{n}$, hence it lifts to a morphism $h_{n+1}\ : \ P'_{n+1}\rightarrow P'_{n+2}$. Thus, the morphism of projective resolutions turned out to be null-homotopic, hence the embedding $M'$ into $M$ is a zero morphism. The obtained contradiction shows that $M$ is simple. 

    By a dual argument, it is easy to show that $N\cong \Ext{2}(M,A)$ is also simple.

    ($\Leftarrow$) Let $M\in Def(\mathcal{E})$ be a simple object. Then any non-zero morphism into it is an epimorphism. Let $\varphi \ : \ M' \twoheadrightarrow M$ be an arbitrary non-zero epimorphism, and let $P_{Z'}$ be a projective cover of $M'$. Then we have an epimorphism from the projective cover $P' \twoheadrightarrow M$, hence a map $P_Z \rightarrow P_{Z'}$, thus $Z'$ splits off from $Z$ as a direct summand. Therefore, if $\alpha\ : \ Z'\rightarrow Z$ is not a split epimorphism, then $\Hom(P_{Z'},M) =0$, so $P_\alpha$ factors through the kernel of the map $P_Z \rightarrow M$, i.e., lifts to a map $P_{Z'}\rightarrow P_Y$. Then in the original category, we have the required splitting morphism $Z' \rightarrow Y$. By a dual argument, it is easy to show that the simplicity of $\Ext{2}(M,A)\cong N$ ensures almost splitness on the left.
\end{proof}

Using the remark on Serre subcategories in categories of finite length with finite number of simple modules, we get that the set of  the Serre subcategories in $\rmod\mathcal{E}$ is in one-to-one correspondence with the set of collections of simple modules. These simple modules are in one-to-one correspondence with the indecomposables in $\mathcal{E}$. Finally, we obtain the theorem, giving an exhaustive description of exact structures on a Krull-Schmidt category of finite type that has Auslander-Reiten sequences:
\begin{theorem}
    Let $\mathcal{E}$ be a Krull-Schmidt category of finite type. Then exact structures on it are in one-to-one correspondence with subsets of
    \begin{enumerate}
        \item The set of almost split sequences in $\mathcal{E}$.
        \item The set of 2-regular simple modules (defects that are simple modules).
    \end{enumerate}
\end{theorem}

In the case where the category $\mathcal{E}$ is not of finite type, knowledge of Auslander-Reiten sequences turns out to be insufficient for an exhaustive description of all exact structures on this category.

\begin{definition}
    Let $\mathcal{E} = (\mathcal{E},\epsilon)$ be an exact category. An object $P$ is called \textit{projective} in this exact structure if for any deflation $d:\ A\longrightarrow B$ and any morphism $f: \ X\longrightarrow B$, there exists a morphism $\varphi: \ P\longrightarrow A$ such that the following diagram commutes:
    $$\xymatrix{
    & P \ar[d]_{f} \ar @{-->}[ld] _{\varphi}\\
    A \ar [r]^{d} & B\\
    }$$
\end{definition}

It is obvious that any object that is projective in the usual sense is projective in the sense of the exact structure. In the case where the category $\mathcal{E}$ is quasi-abelian, an object is projective in the standard exact structure if and only if it is projective in the usual sense.

Now we formulate a theorem that establishes a connection between the abelian envelopes of exact categories, the Serre subcategories of defects of exact structures, and the lattice of exact structures on an additive category.

\begin{theorem}
\label{serre}
    Let $\mathcal{E}$ be an additive category with split idempotents that has all kernels, and let $\epsilon_0\subset \epsilon_1$ be two distinct exact structures on $\mathcal{E}$. Also let the exact category $\mathcal{E}_0 = (\mathcal{E},\epsilon_0)$ have an abelian envelope $\mathcal{A}_r(\mathcal{E}_0)$, and let the canonical functor $i_{\mathcal{E}_0}$ be faithful, full, and exact. Then the exact category $\mathcal{E}_1 = (\mathcal{E},\epsilon_1)$ has an abelian envelope $\mathcal{A}_r(\mathcal{E}_1) = \mathcal{A}_r(\mathcal{E}_0)/Def_{\epsilon_1}\mathcal{E}$.
\end{theorem}
\begin{proof}
    Let $I:\ \mathcal{E}_1\rightarrow\mathcal{E}_1$ be the exact functor that is identity on objects but changes the exact structure. Let $F:\ \mathcal{E}_1\rightarrow \mathcal{B}$ be a right exact functor into some abelian category $\mathcal{B}$. If $X\xrightarrow{i} Y\xrightarrow{d} Z$ is a conflation lying in $\epsilon_1$ but not in $\epsilon_0$, its defect in $\mathcal{A}_r(\mathcal{E}_0)$ identifies with $\coker i_{\mathcal{E}_0}(d)$. From the classification theorem of exact structures, we know that this subcategory is a Serre subcategory. Consequently, the quotient by it will also be abelian.
    $$\xymatrix{
    & Def_{\epsilon_1}\mathcal{E} \ar @{^{(}->}[d]\\
    \mathcal{E}_0 \ar[d]_{I} \ar @{^{(}->}[r]^{i_{\mathcal{E}_0}} & \mathcal{A}_r(\mathcal{E}) \ar @{->>} [d]_{Q} \ar @{-->}[ddl]_{\Tilde{F}}\\
    \mathcal{E}_1 \ar[d]_{F} & \mathcal{A}_r(\mathcal{E})/Def_{\epsilon_1}\mathcal{E} \\
    \mathcal{B} & \\
    }$$

    Since the functor $I$ is exact, $F$ is right exact, $FI$ is also right exact. Hence, by the definition of a right abelian envelope, there exists a unique up to natural isomorphism right exact functor $\Tilde{F}:\ \mathcal{A}_r(\mathcal{E})\rightarrow \mathcal{B}$ such that $FI\cong \Tilde{F}$. Again, let $X\xrightarrow{i} Y\xrightarrow{d} Z$ be a conflation lying in $\epsilon_1$ but not in $\epsilon_0$. Then under the action of the functor $F$, it maps to an exact sequence
    $$F(X)\longrightarrow F(Y) \longrightarrow F(Z) \longrightarrow 0$$
    On the other hand, 
    $$i_{\mathcal{E}_0}(X)\longrightarrow i_{\mathcal{E}_0}(Y) \longrightarrow i_{\mathcal{E}_0}(Z) \longrightarrow M \longrightarrow 0,$$
    where $M$ is the defect of the conflation $X\xrightarrow{i} Y\xrightarrow{d} Z$. Consequently, $\Tilde{F}(M)=0$ for any defect $M$. Hence, by the universal property of the quotient category, there exists a functor $F':\ \mathcal{A}_r(\mathcal{E})/Def_{\epsilon_1}\mathcal{E}\rightarrow \mathcal{B}$ such that $\Tilde{F} = F'\circ Q$. Due to the uniqueness of all choices up to natural isomorphism, we obtain an equivalence of categories $\mathcal{R}ex(\mathcal{E}_1, \mathcal{B})\cong \mathcal{R}ex(\mathcal{A}_r(\mathcal{E})/Def_{\epsilon_1}\mathcal{E}, \mathcal{B})$.

    The functor $I^{-1}:\ \mathcal{E}_1\rightarrow \mathcal{E}_2$ is obviously defined identically on objects and morphisms, but is no longer exact. Nevertheless, the functor $Q\circ i_{\mathcal{E}_0}\circ I^{-1}$ is right exact by construction. There is a natural embedding $\mathcal{R}ex(\mathcal{E}_1,\mathcal{B})\subset \mathcal{R}ex(\mathcal{E}_0,\mathcal{B})$, induced by the exact functor $I$. On the other hand, any right exact functor from $\mathcal{R}ex(\mathcal{E}_1,\mathcal{B})$ to $\mathcal{B}$ extends to a right exact functor from $\mathcal{E}_0$. Moreover, this functor will, by construction, lie in the image of the embedding $\mathcal{R}ex(\mathcal{E}_1,\mathcal{B})$ into $\subset \mathcal{R}ex(\mathcal{E}_0,\mathcal{B})$. Thus, we obtain a functor $i_{\mathcal{E}_1}:\ \mathcal{E}_1\rightarrow \mathcal{A}_r(\mathcal{E})/Def_{\epsilon_1}\mathcal{E}$. Hence, the pair $(\mathcal{A}_r(\mathcal{E})/Def_{\epsilon_1}\mathcal{E}, i_{\mathcal{E}_1})$ is the right abelian envelope of $\mathcal{E}_1$, as required.
\end{proof}

To conclude this chapter, consider the simplest example of a category having exactly 4 exact structures.

\begin{ex}
    Let $\mathcal{E}$ be the category of projective right modules $\rmod A_0$ over the path algebra of the quiver
    \begin{equation*}
        A_0\ : \ \xymatrix{
        \bullet_0 \ar[rr]^\alpha && \bullet_1 \ar[rr]^\beta && \bullet_2 \ar[rr]^\gamma && \bullet_3 \ar[rr]^\delta && \bullet_4 \\
        }
    \end{equation*}
    with relations $\beta\alpha = 0$, $\delta\gamma = 0$. In the category of representations of this quiver, there are exactly two simple objects with projective dimension 2:
    $$0\longrightarrow P_0 \longrightarrow P_1 \longrightarrow P_2 \longrightarrow S_2 \longrightarrow 0$$
    $$0\longrightarrow P_2 \longrightarrow P_3 \longrightarrow P_4 \longrightarrow S_4 \longrightarrow 0$$
    Obviously, they satisfy the 2-regularity conditions, i.e., in the category $\mathcal{E}$ there are 2 almost split sequences:
    \begin{equation*}
        P_0 \longrightarrow P_1 \longrightarrow P_2 \eqno (1)
    \end{equation*}
    \begin{equation*}
        P_2 \longrightarrow P_3 \longrightarrow P_4 \eqno (2)
    \end{equation*}
    The corresponding exact categories (i.e., $\mathcal{E}$ with the exact structure) and the exact structures on them will be denoted respectively as $\mathcal{E}_0 = (\mathcal{E}, \epsilon_0)$, $\mathcal{E}_1=(\mathcal{E}, \epsilon_1)$, $\mathcal{E}_2=(\mathcal{E}, \epsilon_2)$, $\mathcal{E}_{1,2} = (\mathcal{E}, \epsilon_{1,2})$. The category $\mathcal{E}_0$ with the split exact structure is an exact subcategory (of projective objects) of the category $\mathcal{A}_0 = \rmod A$. Then the other exact categories identify with exact subcategories of right modules over, respectively, the algebras:
    \begin{equation*}
        A_1\ : \ \xymatrix{
        \bullet_0 \ar[rr]^\alpha && \bullet_1 \ar[rr]^\gamma && \bullet_2 \ar[rr]^\delta && \bullet_3 \\
        }, \ \delta\gamma = 0
    \end{equation*}
    \begin{equation*}
        A_2\ : \ \xymatrix{
        \bullet_0 \ar[rr]^\alpha && \bullet_1 \ar[rr]^\beta && \bullet_2 \ar[rr]^\gamma && \bullet_3 \\
        }, \ \beta\alpha = 0
    \end{equation*}
    \begin{equation*}
        A_{1,2}\ : \ \xymatrix{
        \bullet_0 \ar[rr]^\alpha && \bullet_1 \ar[rr]^\gamma && \bullet_2 \\
        }, \ \alpha\gamma = 0
    \end{equation*}
    Each of these algebras is simply the endomorphism algebra $\End(\bigoplus P_i)$, where the sum is taken over the projective objects in the given exact structure, i.e., $A_1\cong \End(P_0\oplus P_1 \oplus P_3 \oplus P_4)$, $A_2\cong \End(P_0\oplus P_1 \oplus P_2 \oplus P_3)$, $A_{1,2}\cong \End(P_0\oplus P_1 \oplus P_3)$. Denote $\mathcal{A}_1  = \rmod A_1$, $\mathcal{A}_2  = \rmod A_2$, $\mathcal{A}_{1,2}  = \rmod A_{1,2}$. The categories $\mathcal{A}_{1}$ and $\mathcal{A}_{2}$ are quotient categories of $\mathcal{A}_0$ by the Serre subcategories generated by the objects $S_1$ and $S_2$, respectively; hence they are the right abelian envelopes of the corresponding exact structures. It is easy to see that $\mathcal{E}_{1,2}\cong \mathcal{A}_{1,2}$, where the equivalence is understood in the sense of exact categories. Here, the objects $P_2$ and $P_4$ correspond to the simple objects of the first and second vertices of the quiver of the algebra $A_{1,2}$.

    In fact, this example can be extended by considering the category $\mathcal{P}$ of projective objects of the category $\mathcal{A}_{1,2}$. The category $\mathcal{A}_{1,2}$ contains a unique simple object of global dimension 2, namely $S_2$:
    $$0\longrightarrow P_0 \longrightarrow P_1 \longrightarrow P_2 \longrightarrow S_2 \longrightarrow 0$$
    This means that in the category $\mathcal{P}$, there are 2 distinct exact structures — the split one, coinciding with the one induced from the ambient category, and the maximal one, in which the almost split sequence $P_0\longrightarrow P_1 \longrightarrow P_2$ is a conflation. Its abelian envelope should be a quotient category of $\mathcal{A}_{1,2}$ by the Serre subcategory generated by the simple object of the second vertex $S_2$. It is easy to see that this category is the category of representations of the quiver $A_2$:
    $$\bullet_0 \longrightarrow \bullet_1$$
    In fact, it is obvious that the category $\mathcal{P}$ with the maximal exact structure is simply equivalent to the category of representations of this quiver. Under this equivalence, the object $P_2$ maps to the simple object of the first vertex, and the other projective objects map to the corresponding projective objects of the smaller quiver, so all indecomposable objects of the category of representations of this quiver lie in the image of the embedding of the exact subcategory.

    Further reduction of the algebra in this way is impossible, as it was proved above that for a hereditary algebra (in particular, a path algebra of a quiver without relations), the category of projective modules has only one exact structure (split).
\end{ex}

\section{Derived equivalences of exact structures on $k[x]/(x^n)$}
After obtaining an explicit description of exact structures on additive Krull-Schmidt categories, the following natural question arises: when are two exact structures on the same category derived equivalent? An exhaustive answer to this question is unknown, but we will consider several important examples that admit a nice partial answer.

Let's start with the simplest case, which requires practically no knowledge of deep theory.
\begin{ex}
    Let $\mathcal{E} = \rmod A$, where $A$ is the path algebra of the quiver
    $$\xymatrix{
    0 \ar @/^2ex/[r]^{\alpha} & 1 \ar @/^2ex/[l]^{\beta}\\
    }$$
    $$\alpha \beta = \beta \alpha = 0$$

   It is known that this algebra has infinite global dimension and has exactly 4 indecomposable objects: $P_0$, $P_1$, $S_0$, and $S_1$ (where the projective objects are also injective). The almost split sequences are:
   \begin{gather*}
       0\longrightarrow S_1 \longrightarrow P_0 \longrightarrow S_0 \longrightarrow 0\\
       0\longrightarrow S_0 \longrightarrow P_1 \longrightarrow S_1 \longrightarrow 0
   \end{gather*}
   The abelian envelope of this category with the split exact structure will be the category of modules over the algebra $A_0 = \End(S_0\oplus S_1\oplus P_0\oplus P_1)$. As a quiver, this algebra is represented as follows (for convenience, we label the vertices of the quiver with the corresponding objects of the original category):
   \begin{gather*}
   \xymatrix{
   S_0 \ar[r]^{\gamma} & P_1 \ar[d]^{\delta}\\
   P_0 \ar[u]^{\psi} & S_1 \ar[l]^{\varphi}\\
   }\\
   \delta \gamma = \psi \varphi = 0
   \end{gather*}
   It is easy to check that this category has global dimension 2, and the original category is a quotient category of this one by the Serre subcategory generated by the simple objects of vertices $S_0$ and $S_1$ by theorem \ref{serre}. 

   The abelian envelope of the category $\rmod A_0$ is derived equivalent to the category $\mathcal{E}$ with the split exact structure, since the latter is embedded into it as a subcategory of projective objects, and every object of $\rmod A_0$ is a quotient of a projective one. 

   The category $\mathcal{E}$ admits 2 more exact structures (including one of the two Auslander-Reiten sequences as a conflation), whose abelian envelopes are quotient categories of $\rmod A_0$ by the Serre subcategories generated by the simple object of vertex $S_0$ or $S_1$, respectively. Denote the category $\mathcal{E}$ with the corresponding exact structures as $\mathcal{E}_1$ and $\mathcal{E}_2$. Due to the symmetry of the quiver, these categories will be equivalent. This equivalence is given by the Nakayama functor
   $$\nu(-) = \Hom(-,A)^*,$$
   which swaps projective and injective objects. Since for the algebra $A$, $P_0 = I_1$, $I_1 = P_0$, the Nakayama functor permutes the projective objects. Moreover, there is an obvious embedding $i_0$ of the category $\mathcal{D}^{\flat}(\mathcal{E})$ into $\mathcal{D}^{\flat}(\mathcal{E_0})$ and an embedding $i_1$ of the same category into $\mathcal{D}^{\flat}(\mathcal{E_0})$, sending the projective objects $P_0$ and $P_1$ to the projective objects of the corresponding vertices of the quiver of the algebra $A$. This picture can be represented as a diagram:
   $$\xymatrix{
   \langle S_1 \rangle \ar@{^{(}->}[rd]&& \langle S_0 \rangle \ar@{_{(}->}[ld]\\
   & \mathcal{D}^{\flat}(\rmod A) \ar[ld]^{Q_0} \ar[rd]_{Q_1}& \\
   \mathcal{D}^{\flat}(\mathcal{E}_0) \ar[rd]_{Q_0} && \mathcal{D}^{\flat}(\mathcal{E}_0) \ar[ld]^{Q_1}\\
   & \mathcal{D}^{\flat}(\mathcal{E}_0) \ar@{-->} @/_2ex/ [ul]_{i_0} \ar@{-->} @/^2ex/ [ur]^{i_1}& \\
   }$$
   Then the desired derived equivalence $\Phi$ is given by the composition $$\Phi = i_1\circ \nu \circ Q_0.$$
\end{ex}

\begin{definition}
    An exact category $\mathcal{A}$ is called \textit{smooth} if $\textrm{gl.dim. }\mathcal{A}<\infty$, and \textit{non-smooth (singular)} otherwise.
\end{definition}

\begin{theorem}
\label{dim}
    Let $\mathcal{E}$ be a Krull-Schmidt category of finite type that has Auslander-Reiten sequences. Then the right abelian envelope of its split exact structure $\mathcal{A} = \mathcal{A}_r(\mathcal{E})$ has global dimension 2.
\end{theorem}
\begin{proof}
    This is a simple reformulation of a famous result due to Auslander (see \cite{Aus}, Prop. 5.4.). The category $\mathcal{E}$ embeds into the abelian envelope as a subcategory of projective objects, and any object of the abelian envelope has a projective cover. Let the original category be non-semi-simple. The defects of Auslander -- Reiten sequences lie in the abelian envelope, are simple, and have global dimension 2, hence $\textrm{gl.dim }\mathcal{A}$ is at least 2. On the other hand, the category $\mathcal{E}$ is closed under kernels as a subcategory of $\mathcal{A}$, so the global dimension is at most 2.

    If the original category was split, i.e., all indecomposable objects were projective, then, by the same reasoning, the abelian envelope will also be semi-simple.
\end{proof}

Practically, this theorem means that the abelian envelope of a Krull-Schmidt category of finite type, taken with the split exact structure, will always be smooth. In a sense, the constructed abelian category is a noncommutative resolution of the original (possibly non-smooth) category. In what follows, we will often refer to passing to the abelian envelope of an exact structure as a resolution of the category.

Now we turn to the central example of this work.

Let $A = k[x]/(x^3)$ be the quotient algebra of the polynomial algebra over an algebraically closed field of characteristic 0. It is the path algebra of a quiver with relations:
$$
\xymatrix{
\bullet \ar@(ur,ul)[]_\alpha
}, \ \alpha^3 = 0
$$
Let $\mathcal{E} = \rmod A$ be the category of right $A$-modules. In this category, there are exactly 3 indecomposable objects: $k$, $k[x]/(x^2)$, and $k[x]/(x^3)$, the last of which is projective (free) (this is well-known, for example see \cite{Aus}, chapter VI.2). Consequently, in the category $\mathcal{E}$, there are 2 Auslander-Reiten sequences:
\begin{equation}
    \label{AR1}
    0\longrightarrow k \longrightarrow k[x]/(x^2) \longrightarrow k \longrightarrow 0
\end{equation}
\begin{equation}
    \label{AR2}
    0\longrightarrow k[x]/(x^2) \longrightarrow k\oplus k[x]/(x^3) \longrightarrow k[x]/(x^2) \longrightarrow 0
\end{equation}
Thus, the category $\mathcal{E}$ admits $2^2 = 4$ distinct exact structures. Consider the abelian envelope of $\mathcal{E}$ with the split exact structure. This is none other than the category of right modules over the endomorphism algebra $A_{max}$ of the sum of projective objects. As a quiver, this algebra can be represented as follows:

\begin{gather*}
\xymatrix{
\bullet _0 \ar @/_1pc/ [r] _{\alpha_1} & \bullet _1 \ar @/_1pc/ [l] _{\beta_1} \ar @/_1pc/ [r] _{\alpha_2} & \bullet _2 \ar @/_1pc/ [l] _{\beta_2}\\
}
\\
\beta_1\alpha_1 = 0, \ \alpha_1 \beta_1 = \beta_2 \alpha_2
\end{gather*}

Let $\mathcal{E}_0$ and $\mathcal{E}_1$ be the exact categories where the conflation is, respectively, only the first or second Auslander-Reiten sequence. Then their abelian envelopes also coincide with the category of modules over the endomorphism algebra of the projective generator. As quivers, these algebras are represented as follows:
\begin{gather*}
A_1:\ \xymatrix{
\bullet _1  \ar @/_1pc/ [r] _{\alpha_2} & \bullet _2 \ar @/_1pc/ [l] _{\beta_2}\\
}
\\
(\alpha_2 \beta_2)^2 = 0
\end{gather*}

\begin{gather*}
A_2: \ \xymatrix{
\bullet _0 \ar @/_1pc/ [r] _{\alpha} & \bullet _1 \ar @/_1pc/ [l] _{\beta} \ar @(ur,dr) ^{\gamma} \\
}
\\
\beta\gamma = 0, \ \gamma \alpha = 0, \ \alpha\beta = \gamma^2
\end{gather*}

On one hand, the representation categories $\mathcal{A}_i = \rmod\End(A_i),\ i = 0,1$ are quotient categories of $\mathcal{A}_{max} = \rmod\End(A_{max})$ by the Serre subcategories generated by the simple objects $S_0$ and $S_1$, i.e., the defect of the Auslander-Reiten sequence that is a conflation in the given exact structure. On the other hand, the quotient functor admits an explicit description in internal terms of the category:
$$\xymatrix{
\langle S_0\rangle \ar@{^{(}->}[rr] && \mathcal{A}_{max} \ar@{->>}[rr]^{\Hom(P_1\oplus P_2,-)}  && \mathcal{A}_0
}$$
$$\xymatrix{
\langle S_1\rangle \ar@{^{(}->}[rr] && \mathcal{A}_{max} \ar@{->>}[rr]^{\Hom(P_0\oplus P_2,-)}  && \mathcal{A}_1
}$$

Similarly, the factorization of the category $\mathcal{A}_{max}$ by the Serre subcategory generated by $S_0$ and $S_1$ yields the original category $\mathcal{E}$:
\begin{equation}
\label{factor}
\xymatrix{
\langle S_0, S_1\rangle \ar@{^{(}->}[rr] && \mathcal{A}_{max} \ar@{->>}[rr]^{\Hom(P_2,-)}  && \mathcal{E}
}
\end{equation}
All these functors extend to functors between derived categories and have left adjoints. As a result, we have the following diagram:
\begin{equation}
\label{flop}
        \xymatrix{
        && D^{\flat}(\rmod A_{max}) \ar @/^1pc/ [ldld] ^{RHom(P_1\oplus P_2, -)} \ar @/_1pc/ [rdrd] _{RHom(P_0\oplus P_2, -)} && \\
        \\
        D^{\flat}(\rmod A_{0}) \ar @/^1pc/ [rdrd] ^{RHom(P_2, -)} \ar @/^1pc/ [ruru]  && && D^{\flat}(\rmod A_{1}) \ar @/_1pc/ [ldld] _{RHom(P_2,-)} \ar @/_1pc/ [lulu] \\
        \\
        && D^{\flat}(\rmod A) \ar @/^1pc/ [lulu]   \ar @/_1pc/ [ruru]  &&\\
        }
\end{equation}

By Theorem \ref{dim}, the category $\mathcal{A}_{max}$ has global dimension 2, whereas the global dimension of the other categories is infinite. Therefore, diagram (\ref{flop}) can be viewed as a noncommutative resolution of the (non-smooth) category $\mathcal{E}$. This situation resembles a similar situation in geometry — the Atiyah flop. It is known that for threefolds, the composition of functors passing through the upper resolution yields an equivalence of the side categories. Unfortunately, it is easy to check that this is not the case here. Nevertheless, the following theorem holds:
\begin{theorem}
    The categories $D^{\flat}(\mathcal{A}_0)$ and $D^{\flat}(\mathcal{A}_1)$ are equivalent. 
    \label{THM}
\end{theorem}

Let us introduce a few definitions and formulate an important theorem that will allow us to construct autoequivalences of the derived category and proof the theorem.
    \begin{definition}
        An object $S\in \mathcal{D}^{\flat}(\mathcal{A})$ is called n-spherical if the following conditions hold:
        \begin{enumerate}
            \item The object $S$ has finite projective dimension.
            \item $\RHom(S,E)$ and $\RHom(E,S)$ are finite-dimensional for any $E\in \mathcal{A}$.
            \item $\Ext{i}(S,S) = k$ for $i = 0,n$ and $0$ otherwise.
            \item The canonical pairing $\Ext{i}(E,S)\times \Ext{n-i}(S,E)\rightarrow \Ext{n}(S,S) = k$ is non-degenerate.
        \end{enumerate}
    \end{definition}

    \begin{definition}
        Let $S$ be a spherical object in $\mathcal{D}^{\flat}(\mathcal{A})$. \textit{}{Spherical twist} is the functor $T_S:\ \mathcal{D}^{\flat}(\mathcal{A}) \rightarrow \mathcal{D}^{\flat}(\mathcal{A})$ sending an object $E$ to the cone of the canonical morphism:
        $$\RHom(S,E)\otimes S \xrightarrow{eval} E \rightarrow C(eval)\cong T_S (E),$$
        where $eval$ is the canonical pairing (evaluation of the corresponding element from $\RHom$).
    \end{definition}

    \textit{Remark.} In a triangulated category, the cone is almost never functorial (it is functorial if and only if the triangulated category is semi-simple, as proven by Verdier); therefore, from the point of view of triangulated categories, the spherical twist $T_S$ is defined only for objects. Nevertheless, this technical difficulty can be circumvented by considering the derived category not as a triangulated category, but as a dg-category. In this case, the cone can be defined functorially. For details, see \cite{BK}.

    \begin{definition}
        A collection of n-spherical objects $(S_0,\ ... \ , S_m)$ is said to be in an $A_m$-\textit{configuration} if 
        $$\textrm{dim }\RHom(S_i,S_j) = \left\{ \begin{aligned}
            1, \ |i-j| = 1,\\
            0, \ |i-j|\geq 2
        \end{aligned}\right.
        $$
    \end{definition}

    \begin{theorem}[\cite{ST}, Theorem 2.17]
    \label{sph}
        Let $(S_0,\ ... \ , S_m)$ be a collection of n-spherical objects in $\mathcal{D}^\flat (\mathcal{A})$ in an $A_m$-configuration.
        \begin{enumerate}
            \item Each spherical twist $T_{S_i}$ determines an autoequivalence of the category $\mathcal{D}^{\flat}(\mathcal{A})$.
            \item The spherical twists $T_{S_i}$ satisfy the relations of the braid group $B_{m}$ up to a graded natural isomorphism, i.e., the braid group $B_{m-1}$ acts on $\mathcal{D}^{\flat}(\mathcal{A})$.
            \item If $n\geq 2$, this action is effective.
        \end{enumerate}
    \end{theorem}
\begin{proof}
    (of the theorem \ref{THM}). The functor $i^*_0$ acts on projective objects in an obvious way, simply sending $P_{\mathcal{A}_0}(i)$ to $P_{\mathcal{A}_{max}}(i)$. Similarly, the functor $\pi_{1*}$ acts elementarily on objects $P_{\mathcal{A}_{max}}(0)$ and $P_{\mathcal{A}_{max}}(2)$. Then, due to the exactness of the functors going downwards and the right exactness of the functors going upwards, we only need to keep track of the images of projective modules.

    To find the required equivalence, let us try to find an autoequivalence of the maximal resolution -- the category $\mathcal{D}^{\flat}(\mathcal{A}_{max})$ -- that swaps the simple objects $S_0$ and $S_1$. Let us write down explicitly the Ext algebra of the simple objects $S_0$ and $S_1$ of the category $\mathcal{A}_{max}$. They can be computed directly (using projective resolutions), or one can use the fact that the relations in the quiver of the algebra $A_{max}$ are of degree 2, hence the $\Ext{}$'s between simple objects correspond to relations between vertices. In any case, $\Ext{1}(S_0,S_1) = \Ext{1}(S_i,S_0) = k$, $\Ext{2}(S_i,S_i) = k$, and all others are 0, since the global dimension of the algebra $A_{max}$ is 2. Thus, the $\Ext{}$ algebra has the form:
    \begin{gather*}
    \xymatrix{
    S_0 \ar@/^2ex/[rr]^{a[1]}&& S_1 \ar@/^2ex/[ll]^{b[1]}\\
    }\\
    aba = bab = 0
    \end{gather*}

    The objects $S_0$ and $S_1$ are $2$-spherical, which can be checked elementarily. Hence, they define autoequivalences $T_0$ and $T_1$ of the category $\mathcal{D}^{\flat}(\mathcal{A})$ — the twists with respect to the spherical objects $S_0$ and $S_1$, respectively. These functors satisfy the braid group relations: $T_0 T_1 T_0 = T_1 T_0 T_1$. Let $\Phi = T_0 T_1 T_0 = T_1 T_0 T_1$. We show that the functor $\Phi[-1]$ swaps the objects $S_0$ and $S_1$ in $\mathcal{D}^{\flat}(\mathcal{A})$. Indeed, $T_i(S_i) = S_i[1]$, $T_0(S_1) = M_{10}$, $T_1(S_0) = M_{01}$, where the modules $M_{10}$ and $M_{01}$ are structured as follows:
    \begin{gather*}
    M_{10} = \xymatrix{k \ar [rr]^{1} && k \ar [rr]^{0} && 0}\\
    M_{01} = \xymatrix{k && k \ar [ll]^{1} && 0 \ar [ll]^{0}}   
    \end{gather*}
    (for convenience, zero arrows are omitted). Now $T_1(M_{10}) = S_0$, $T_0(M_{01}) = S_1$. Thus, $T_0 T_1 T_0 (S_1) = S_0[1]$, $T_1 T_0 T_1 (S_0) = S_1[1]$, as required.

    Let $Q_i = \RHom(P_i\oplus P_2, -)$. Then the kernel of the composition $Q_i \Phi$ is the subcategory generated by the object $S_{1-i}$. Hence, this composition factors uniquely through the corresponding quotient category. This gives the required equivalence of $\mathcal{D}^{\flat}(\mathcal{A}_0)$ and $\mathcal{D}^{\flat}(\mathcal{A}_1)$.
\end{proof}

It is not difficult to generalize the results of this example to the more general case $A = k[x]/(x^n)$. In this case, we have $n-1$ Auslander-Reiten sequences of the form
\begin{gather*}
    0 \rightarrow k \rightarrow k[x]/(x)\rightarrow k\rightarrow 0,\ \\
    0 \rightarrow k[x]/(x^m) \rightarrow k[x]/(x^{m-1})\oplus k[x]/(x^{m+1})\rightarrow k[x]/(x^m)\rightarrow 0,\\ 1<m<n
\end{gather*}
Thus, the category $\mathcal{A} = \rmod A$ admits $2^{n-1}$ exact structures. Let $C\subset \{1..n-1\}$. Denote by $\mathcal{A}_C$ the abelian envelope of the exact category $\mathcal{A}$ with the exact structure in which the conflations are those AR-sequences that start and end at the module $k[x]/(x^k)$, $k\in C$. The maximal resolution of the category $\mathcal{A}$, the category $\mathcal{A}_{max} = \mathcal{A}_\emptyset$, is the category of modules over the path algebra of the quiver with relations
 \begin{gather*}
        \xymatrix{
\bullet _0 \ar@<-1ex> [r] _{\alpha_1} & \bullet _1 \ar@<-1ex> [l] _{\beta_1} \ar@<-1ex>  [r] _{\alpha_2} & ... \ar@<-1ex>  [l] _{\beta_2} \ar@<-1ex> [r] _{\alpha_{n-1}} & \bullet _{n-1} \ar@<-1ex> [r] _{\alpha_n} \ar@<-1ex> [l] _{\beta_{n-1}} & \bullet _n \ar@<-1ex> [l] _{\beta_{n}}
}\\
\beta_1\alpha_1 = 0, \ \alpha_i\beta_i = \beta_{i+1}\alpha_{i+1}
\end{gather*}

This category has global dimension 2. Then the Ext algebra of the objects $(S_0,\ ...\ , S_{n-1})$ has the form:
\begin{gather*}
        \xymatrix{
        \bullet _{S_0} \ar @/^1.2pc/ [r] ^{a_1} _{[1]} & \bullet_{S_1} \ar @/^1.2pc/ [l] ^{[1]} _{b_1} \ar @/^1.2pc/ [r] ^{a_2} _{[1]}& ... \ar @/^1.2pc/ [l] ^{[1]} _{b_{2}} \ar @/^1.2pc/ [r] ^{a_{n-1}} _{[1]} & \bullet_{S_{n-1}} \ar @/^1.2pc/ [l] ^{[1]} _{b_{n-1}}  \\
        } \\
        ab+ba = 0, \ a^2 = b^2 = 0
\end{gather*}
(for readability, indices are omitted in the relations, as they are obvious from the context). It is clear that these objects are 2-spherical and are in an $A_{n-1}$ configuration. Thus, by Theorem \ref{sph}, the braid group $B_{n-1}$ acts on the category $\mathcal{D}^{\flat}(\mathcal{A}_{max})$. Let $\rho: \ B_{n-1}\rightarrow \textrm{Autoeq}(\mathcal{D}^{\flat}(\mathcal{A}_{max}))$ be an injective homomorphism (since the action is effective).
\begin{lemma}
    For any $i,j\in \{1..n-1\}$, there exists an autoequivalence $\Phi_{ij}\in \mathrm{Im }\rho$ such that $\Phi_{ij}(S_i) = S_j$, $\Phi_{ij}(S_j) = S_i$.
\end{lemma}
\begin{proof}
    Without loss of generality, let $i<j$. We prove that $\Phi_{ij}$ can be represented as a composition $\Phi_{ij} = T_i T_{i+1} ... T_{j-1} T_{j} T_{i}...T_{j-1}...T_{i} T_{i+1} T_i = \prod_{p = 0}^{j-i} \prod_{q=i}^{j-p} T_{q}$. Elements of this form are the images, under the embedding $\rho$, of elements of the braid group whose squares generate its center. By the braid group relations, this functor can also be represented, for example, as $\Phi_{ij} = \prod_{p = i}^{j} \prod_{q=0}^{j-p} T_{j-q}= T_j T_{j-1} ... T_{i+1} T_{i} T_{j}...T_{i+1}...T_{j} T_{j-1} T_j$. 

    Denote by $M_{ij}$ the module over $A_{max}$ having the form
     \begin{gather*}
    M_{ij} = \xymatrix{0 \ar[r]^{0} & ... \ar[r]^{0} & 0 \ar[r]^{0}&k \ar[r]^{1} & ...\ar[r]^{1} & k \ar [r]^{0} & 0 \ar[r]^{0} & ... \ar[r]^{0}& 0}, \ i>j\\
    M_{ij} = \xymatrix{0 & ... \ar[l]^{0} & 0 \ar[l]^{0} & k \ar[l]^{0}&... \ar[l]^{1}& k \ar [l]^{1} & 0 \ar[l]^{0} & ... \ar[l]^{0} & 0 \ar[l]^{0} }, \ i<j,   
    \end{gather*}
    where the one-dimensional vector spaces are located at vertices from i-th to j-th (most of the zero operators in the figure are omitted for readability). In particular, $M_{ii} = S_i$. By a direct computation, it is elementarily verified that 
    \begin{gather*}
        T_j T_{j-1} ... T_i(S_i) = M_{ij},\\
        T_j T_{j-1} ... T_{i+1}(M_{ij}) = M_{jj} = S_j,\\
        T_p(S_q) = S_q, \ |p-q|>1\\
        T_i T_{i+1} ... T_j(S_j) = M_{ji},\\
        T_i T_{i+1} ... T_{j-1}(M_{ji}) = M_{ii} = S_i,\\
    \end{gather*}
    The statement of the lemma then follows obviously.
\end{proof}

It is also easy to verify that the functor $\Phi_{ij}$ not only swaps $S_i$ and $S_j$, but also for any $k$: $0\leq k \leq j-i$, we have $\Phi_{ij}(S_{i+k}) = S_{j-k}$. However, it "spoils" the modules $S_{i-1}$ and $S_{j+1}$, acting on them as follows:
\begin{gather*}
    \Phi_{ij}(S_{i-1}) = M_{i-1,j}\\
    \Phi_{ij}(S_{j+1}) = M_{i,j+1}
\end{gather*}

From this, it automatically follows that it is always possible to permute groups consisting of the same number of consecutive $S_i$'s, provided these groups are separated by at least one simple object. This yields the main theorem of this chapter, giving a sufficient condition for the derived equivalence of two exact structures on $\rmod k[x]/(x^n)$. In fact, this is a criterion:
\begin{theorem}
\label{main}
    Let $C_1$, $C_2\subset \{0..n-1\}$, and let $\mathcal{E}_{C_i}$ be the category $\rmod k[x]/(x^n)$ with the exact structure corresponding to $C_i$.
    The bounded derived categories $\mathcal{D}^{\flat}(\mathcal{E}_{C_1})$ and $\mathcal{D}^{\flat}(\mathcal{E}_{C_2})$ are equivalent iff $C_1$ and $C_2$ consist of the same set of consecutive groups of elements, each of these groups is separated from the other by at least one number, and the unordered sets of lengths of these groups coincide.
\end{theorem}
\begin{proof}
    Let us show the necessity. Consider the Euler bilinear form $\chi$ on $K_0(\mathcal{D}^{\flat}(\mathcal{A}_{max}))$. In the basis of projectives, it has the following matrix:
    \begin{equation}
    \label{matrix}    
        \chi = \begin{pmatrix}
            1 & 1 & 1 & \dots & 1\\
            1 & 2 & 2 & \dots & 2\\
            1 & 2 & 3 & \dots & 3\\
              &   & \dots & & \\
            1 & 2 & 3 &\dots & n\\
        \end{pmatrix}
    \end{equation}
    Using Gauss elimination method for bilinear forms, it can be easily checked, that this form is diagonalisable, and is equivalent (as an integral form) to identity matrix. Similarly, the bilinear form $\chi_C$, defined on $K_0(\mathcal{D}^{\flat}(\mathcal{E}_{C}))$, where C is as in the theorem, will have matrix, that can be obtained from matrix \ref{matrix} by removing columns and rows with numbers from $C$. Let $J = [0, n]-C$, and $l_k = j_k - j_{k-1}$ for $k\in [1, \#C+1]$ (these are the lengths of groups of elements in C). Again, from elementary Gauss elimination process, it can be seen, that these matrices are diagonalisable, and the diagonal elements are exactly $l_i$. Hence we get that these forms are equivalent as integral bilinear forms iff the collections of $l_k$ coincide, and the corresponding derived cateegories can be equivalent only in this case. This finishes the proof.
\end{proof}

\section{Conclusion}
The results obtained in this work can be generalized in various directions. First of all, it is of interest to formulate a necessary and sufficient condition for the derived equivalence of two exact structures on an additive category. It is unknown whether all such derived equivalences factor through the derived category of a smooth resolution of the category.

Another direction for further research is obtaining an explicit classification of exact structures on Krull-Schmidt categories that are not of finite type. Not all Serre subcategories in categories of modules over such categories can be described by their simple objects. This question is related to the problem of classifying (subcanonical) Grothendieck topologies on the category, and studying topology on the Ziegler spectrum of a category.

The author expresses deep gratitude to his scientific advisor, PhD A.I. Bondal. 
The work is supported by the state assignment of MIPT (project FSMG-2026-0015).


\begin{thebibliography}{99}
\bibitem[Aus]{Aus} M. Auslander, I. Reiten, S. O. Smalo. Representation theory of Artin algebras, Comm. algebra 3 (295).
\bibitem[BVdB]{BvB} A. I. Bondal and M. Van den Bergh. Generators and representability of functors in commutative and noncommutative geometry. Moscow Mathematical Journal, 3(1):1–36, 2003.
\bibitem[BB]{BB} A. Bodzenta and A. Bondal. Highest weight categories as abelian envelopes of exact categories.
\bibitem[BK]{BK} A.I. Bondal, M.M. Kapranov. Enhanced triangulated categories. Math. USSR Sbornik, vol. 70 (1991).
\bibitem[Bühler]{Buh} T. Bühler. Exact categories. Expo. Math. 28 (2010).
\bibitem[Crivei]{Cr} S. Crivei. Maximal exact structures on additive categories revisited.
\bibitem[En]{En} H. Enomoto. Classification of exact stuctures and Cohen--Macaulay-finite algebras. Advances in Mathematics, 335 (2018).
\bibitem[Sch]{Sch} J.-P. Schneiders. Quasi-abelian categories and sheaves. Mem. Soc. Math. Fr.(N.S.), (76):vi+134, 1999.
\bibitem[ST]{ST} P. Seidel and R. Thomas. Braid group actions on derived categories of coherent sheaves. Duke mathematical journal, vol. 108, 2001.
\bibitem[Rump]{Rump} W. Rump. On the maximal exact structure of an additive category. FUNDAMENTA MATHEMATICAE 214 (2011).
\end{thebibliography}
\end{document}